\documentclass[a4paper]{article}

\usepackage{bm}
\usepackage{amsmath}
\usepackage{amsthm}
\usepackage{amssymb}
\usepackage{hyperref}
\usepackage{cleveref}
\usepackage{xparse}
\usepackage{dsfont}
\usepackage{tocbibind} 
\usepackage{enumitem}
\usepackage{caption}

\hypersetup{bookmarksnumbered}
\numberwithin{equation}{section}

\usepackage[lmargin=3cm,rmargin=3cm,tmargin=3cm,bmargin=3.5cm]{geometry}

\usepackage{lastpage}
\usepackage{fancyhdr}
\usepackage[ddmmyyyy,hhmmss]{datetime}
\newdateformat{daymonthyeardate}{%
  \THEDAY.~\monthname[\THEMONTH] \THEYEAR}
\fancypagestyle{plain}{
\fancyhf{}
\rfoot{}
\cfoot{\footnotesize Page \thepage{} of \pageref{LastPage}}
}
\newtheorem{theorem}{Theorem}[section]

\newtheorem{lemma}[theorem]{Lemma}

\theoremstyle{remark}
\newtheorem{remark}[theorem]{Remark}
\newtheorem*{rem*}{Remark}

\usepackage{todonotes}

\newcommand{\Lip}{\operatorname{Lip}}
\newcommand{\LebesgueMeasure}{\boldsymbol{\lambda}}

\renewcommand{\emptyset}{\varnothing}

\newcommand{\LayerFunc}{\boldsymbol{\ell}}
\newcommand{\CoeffFunc}{\boldsymbol{c}}
\newcommand{\Architecture}{\boldsymbol{a}}
\NewDocumentCommand\PreNorm{O{p}D<>{\alpha}}{\Gamma_{#2,#1}}
\NewDocumentCommand\NNSigma{O{n}D<>{\LayerFunc}O{\CoeffFunc}}{\Sigma_{#1}^{#2, #3}}
\NewDocumentCommand\ApproxSpace{O{p}D<>{\CoeffFunc}O{\alpha}}{A^{#3,#1}_{\LayerFunc,#2}}

\NewDocumentCommand\UnitBall{O{p}D<>{\CoeffFunc}O{\alpha}}{U^{#3,#1}_{\LayerFunc,#2}}
\NewDocumentCommand\ClosedUnitBall{O{p}D<>{\CoeffFunc}O{\alpha}}{\overline{U}^{#3,#1}_{\LayerFunc,#2}}

\newcommand{\WeightSize}[1]{\| #1 \|_{\mathcal{NN}}}

\newcommand{\N}{\mathbb{N}}

\newcommand{\R}{\mathbb{R}}

\newcommand{\indicator}{\mathds{1}}

\newcommand{\deterministic}{\mathrm{det}}
\newcommand{\MonteCarlo}{\mathrm{MC}}

\newcommand{\Alg}{\operatorname{Alg}}

\makeatletter
\newcommand{\oset}[3][0ex]{%
  \mathrel{\mathop{#3}\limits^{
    \vbox to#1{\kern-2\ex@
    \hbox{$\scriptstyle#2$}\vss}}}}
\makeatother

\makeatletter
\newcommand\avsuminner[2]{%
  {\sbox0{$\m@th#1\sum$}%
   \vphantom{\usebox0}%
   \ooalign{%
     \hidewidth
     \smash{\vrule height\dimexpr\ht0+1pt\relax depth\dimexpr\dp0+1pt\relax}%
     \hidewidth\cr
     $\m@th#1\sum$\cr
   }%
  }%
}
\makeatother

\newcommand{\x}{\boldsymbol{x}}

\newcommand{\din}{d_{\mathrm{in}}}
\newcommand{\dout}{d_{\mathrm{out}}}

\newcommand{\nrow}[1]{%
  \relax
  \vcenter to 0pt{%
    \vss
    \kern-1.5ex
    \rlap{$\left.\vphantom{\begin{matrix}0\\\vdots\\0\end{matrix}}\kern3em\right\}#1$}%
    \vss
  }%
}

\binoppenalty=\maxdimen
\relpenalty=\maxdimen
\makeatletter
\newcommand{\subjclass}[2][1991]{%
  \let\@oldtitle\@title%
  \gdef\@title{\@oldtitle\footnotetext{#1 \emph{Mathematics subject classification.} #2}}%
}
\newcommand{\keywords}[1]{%
  \let\@@oldtitle\@title%
  \gdef\@title{\@@oldtitle\footnotetext{\emph{Key words and phrases.} #1.}}%
}
\makeatother

\title{Sobolev-type embeddings\\{} for neural network approximation spaces}

\keywords{Deep Neural Networks,
Approximation Spaces,
Hölder spaces,
Embedding theorems,
Optimal Learning Algorithms}

\subjclass[2020]{Primary: 68T07, 46E35 Secondary: 65D05, 46E30}
\date{}

\author{Philipp Grohs%
\thanks{Faculty of Mathematics,
University of Vienna,
Oskar-Morgenstern-Platz~1,
A-1090 Vienna, Austria}
\,\thanks{Research Network Data Science @ Uni Vienna,
Währinger Straße 29/S6,
A-1090 Vienna, Austria}
\,\thanks{Johann Radon Institute,
Altenberger Straße 69,
A-4040 Linz,
Austria}
\,\footnotemark[5]
\, and Felix Voigtlaender\thanks{Department of Mathematics,
Technical University of Munich,
Boltzmannstr.\ 3,
85748 Garching bei München,
Germany}
\,\footnotemark[1]\,\,\,\thanks{Both authors contributed equally to this work.}
\,\,\thanks{F.\ Voigtlaender acknowledges support by the German Research Foundation (DFG)
in the context of the Emmy Noether junior research group VO 2594/1--1.}}

\begin{document}
\maketitle
\begin{quote}
\begin{center}
    \emph{This paper is dedicated to Ron DeVore on the occasion of his 80th birthday}
\end{center}
\end{quote}

\begin{abstract}
  We consider neural network approximation spaces that classify functions
  according to the rate at which they can be approximated (with error measured in $L^p$)
  by ReLU neural networks with an increasing number of coefficients,
  subject to bounds on the magnitude of the coefficients and the number of hidden layers.
  We prove embedding theorems between these spaces for different values of $p$.
  Furthermore, we derive sharp embeddings of these approximation spaces into Hölder spaces.
  We find that, analogous to the case of classical function spaces
  (such as Sobolev spaces, or Besov spaces) it is possible to trade ``smoothness''
  (i.e., approximation rate) for increased integrability. 
  
  Combined with our earlier results in [arXiv:2104.02746], our embedding theorems
  imply a somewhat surprising fact related to ``learning'' functions
  from a given neural network space based on point samples:
  if accuracy is measured with respect to the uniform norm,
  then an optimal ``learning'' algorithm for reconstructing functions
  that are well approximable by ReLU neural networks
  is simply given by piecewise constant interpolation on a tensor product grid.
\end{abstract}


\section{Introduction}
\label{sec:Introduction}

We consider approximation spaces associated to neural networks with ReLU activation function
$\varrho:\mathbb{R}\to \mathbb{R},$  $\varrho(x) = \max\{0,x\}$
as introduced in our recent work \cite{GrohsTheoryToPracticeGap};
see also Section~\ref{sec:NNApproximationSpaces}.
Roughly speaking, these spaces $\ApproxSpace$ are defined as follows:
Given a depth-growth function $\LayerFunc \!:\! \mathbb{N} \to \mathbb{N}_{\geq 2} \cup \{\infty\}$
and a coefficient growth function $\CoeffFunc: \mathbb{N} \to \mathbb{N} \cup \{\infty\}$,
the approximation space $\ApproxSpace ([0,1]^d)$ consist of all functions
that can be approximated up to error $\mathcal{O}(n^{-\alpha})$,
with the error measured in the $L^p([0,1]^d)$ (quasi)-norm for some $p\in (0,\infty]$,
by neural networks with at most $\LayerFunc(n)$ layers and $n$ coefficients (non-zero weights)
of size at most $\CoeffFunc(n)$, $n\in \mathbb{N}$.
As shown in \cite{GrohsTheoryToPracticeGap} these spaces can be equipped
with a natural (quasi)-norm; see also Lemma~\ref{lem:ApproximationSpaceProperties} below.
In most applications in machine learning, one is interested in ``learning''
such functions from a finite amount of samples.
To render this problem well defined we would like the space $\ApproxSpace ([0,1]^d)$
to embed into $C([0,1]^d)$.
The question that we consider in this work is therefore to find conditions
under which such an embedding holds true.
More generally, we are interested in characterizing the existence of embeddings
of the approximation spaces into the Hölder spaces $C^{0,\beta}([0,1]^d)$ for $\beta \in (0, 1]$. 
The existence of such embeddings is far from clear; for example, recent work
\cite{devore2020neural,grohs2019deep,daubechies2021nonlinear,NeuralNetworksApproximatingRefinableFunctions}
shows that fractal and highly non-regular functions can be efficiently approximated
by (increasingly deep) ReLU neural networks.

The main contribution of this paper are sharp conditions
for when the space $\ApproxSpace ([0,1]^d)$ embeds into $C([0,1]^d)$ or $C^{0,\beta}([0,1]^d)$.
These results in particular show that, similar to classical Sobolev spaces,
one can trade smoothness for integrability in the context of neural network approximation,
where ``smoothness'' of a function is understood as it being well approximable by neural networks.

Building on these insights we consider the (highly practically relevant)
problem of finding optimal numerical algorithms for reconstructing arbitrary functions
in $\ApproxSpace[\infty] ([0,1]^d)$ from finitely many point samples,
where the reconstruction error is measured in the $L^\infty([0,1]^d)$ norm:
As a corollary to our embedding theorems we establish the somewhat surprising fact that,
irrespective of $\alpha$, an optimal algorithm is given by piecewise constant interpolation
with respect to a uniform tensor-product grid in $[0,1]^d$.  

\subsection{Description of Results}

In the following sections we sketch our main results.
To simplify the presentation, let us assume that the coefficient growth function $\CoeffFunc$
satisfies the asymptotic growth $\CoeffFunc \asymp n^\theta\cdot (\ln(2n))^{\kappa}$
for certain $\theta \geq 0$ and $\kappa \in \mathbb{R}$
(our general results hold for any coefficient growth function).
Let
\[
  \LayerFunc^\ast
  := \sup_{n\in\N}
       \LayerFunc(n)
  \in \N_{\geq 2} \cup \{ \infty \}
  \qquad \mbox{and} \qquad
  \gamma^\ast
  := \theta \cdot \LayerFunc^\ast + \lfloor \LayerFunc^\ast/2\rfloor
  \in [1,\infty].
\]

\subsubsection{Embedding Results}

Our first main theorem provides sharp embedding results into $C([0,1]^d)$.
A slight simplification of our main result reads as follows.

\begin{theorem}\label{thm:intro1}
  Assume that $\alpha > \frac{d}{p}\cdot \gamma^\ast$.
  Then the embedding $\ApproxSpace ([0,1]^d)\hookrightarrow C([0,1]^d)$ holds.
  Assume that $\alpha < \frac{d}{p}\cdot \gamma^\ast$.
  Then the embedding $\ApproxSpace ([0,1]^d)\hookrightarrow C([0,1]^d)$ does \emph{not} hold.
\end{theorem}

Theorem~\ref{thm:intro1} is a special case of Theorems~\ref{thm:SobolevTypeEmbeddingSufficient}
and \ref{thm:SobolevTypeEmbeddingNecessary} below that will be proven in Section~\ref{sec:Embeddings}.

We would like to draw attention to the striking similarity between Theorem~\ref{thm:intro1}
and the classical Sobolev embedding theorem for Sobolev spaces $W^{\alpha,p}$
which states that $W^{\alpha,p}([0,1]^d)\hookrightarrow C([0,1]^d)$ holds if $\alpha > \frac{d}{p}$,
but not if $\alpha < \frac{d}{p}$.
Just as in the case of classical Sobolev spaces we show that,
also in the context of neural network approximation,
it is possible to trade ``smoothness'' (i.e., approximation order) for integrability.

We also prove sharp embedding results into the Hölder spaces $C^{0,\beta}([0,1]^d)$.
A slight simplification of our main result reads as follows.

\begin{theorem}\label{thm:intro2}
  Let $\beta \in (0,1)$.
  \begin{itemize}
    \item If $\alpha > \frac{\beta + \frac{d}{p}}{1-\beta}\cdot \gamma^\ast$,
          then the embedding $\ApproxSpace ([0,1]^d)\hookrightarrow C^{0,\beta}([0,1]^d)$ holds.

    \item If $\alpha < \frac{\beta + \frac{d}{p}}{1-\beta}\cdot \gamma^\ast$,
          then the embedding $\ApproxSpace ([0,1]^d)\hookrightarrow C^{0,\beta}([0,1]^d)$
          does \emph{not} hold.
  \end{itemize}
\end{theorem}

Theorem~\ref{thm:intro2} is a special case of Theorems~\ref{thm:NNIntoHoelderSpace}
and \ref{thm:SobolevTypeEmbeddingNecessary} below that will be proven in Section~\ref{sec:Embeddings}.

Theorems~\ref{thm:intro1} and \ref{thm:intro2} provide a complete characterization
of the embedding properties of the neural network approximation spaces $\ApproxSpace([0,1]^d)$
away from the ``critical'' index $\alpha = \frac{d}{p}\cdot \gamma^\ast$
(embedding into $C([0,1]^d)$), respectively $\alpha = \frac{\beta + \frac{d}{p}}{1-\beta}\cdot \gamma^\ast$
(embedding into $C^{0,\beta}([0,1]^d)$).
We leave the (probably subtle) question of what happens at the critical index open for future work.

\subsubsection{Optimal Learning Algorithms}

Our embedding theorems allow us to deduce an interesting and surprising result
related to the algorithmic feasibility of a learning problem
posed on neural network approximation spaces. 

More precisely, we consider the problem of approximating a function in the unit ball 
\[
  \UnitBall[\infty]([0,1]^d)
  := \big\{
       f \in \ApproxSpace[\infty]([0,1]^d)
       \,\,\colon\,\,
       \| f \|_{\ApproxSpace[\infty]([0,1]^d)}
       \leq 1
     \big\}
\]
of the neural network approximation space $\ApproxSpace[\infty] ([0,1]^d)$
under the constraint that the approximating algorithm
is only allowed to access \emph{point samples} of the functions $f \in U$.
Formally, a (deterministic) algorithm $A : \UnitBall[\infty]([0,1]^d) \to L^\infty([0,1]^d)$
using $m$ point samples is determined by a set of sample points
$\x = (x_1,\dots,x_m) \in ([0,1]^d)^m$ and a map
$Q : \R^m \to L^\infty([0,1]^d)$ such that
\[
  A (f) = Q\bigl(f(x_1),\dots,f(x_m)\bigr)
  \qquad \forall \, f \in \UnitBall[\infty]([0,1]^d) .
\]
Note that $A$ need \emph{not} be an ``algorithm'' in the usual sense,
i.e., implementable on a Turing machine.
The set of all such ``algorithms'' is denoted $\Alg_m (\UnitBall[\infty]([0,1]^d))$
and we define the optimal order for (deterministically) uniformly approximating
functions in $\UnitBall[\infty]([0,1]^d)$ using point samples
as the best possible convergence rate with respect to the number of samples:
\begin{multline*}
  \beta^{\deterministic}_{\ast} (\UnitBall[\infty]([0,1]^d))
  := \sup
     \Big\{
       \lambda \geq 0
       \,\,\colon\,\,
       \exists \, C > 0 \,\,
         \forall \, m \in \N: \\
           \inf_{A \in \Alg_m (\UnitBall[\infty]([0,1]^d))}
             \sup_{f \in \UnitBall[\infty]([0,1]^d)}
               \| A(f) - f \|_{L^\infty([0,1]^d)}
           \leq C \cdot m^{-\lambda}
     \Big\} .
\end{multline*}
In a similar way one can define randomized (Monte Carlo) algorithms
and consider the optimal order $\beta_\ast^{\MonteCarlo} (\UnitBall[\infty]([0,1]^d))$
for approximating functions in $\UnitBall[\infty]([0,1]^d)$ using randomized algorithms
based on point samples; see \cite[Section~2.4.2]{GrohsTheoryToPracticeGap}.
In \cite[Theorem~1.1]{GrohsTheoryToPracticeGap} we showed
(under the assumption $\LayerFunc^\ast \geq 3$) that 
\begin{equation}\label{eq:optimal_sampling_rates}
    \beta^{\deterministic}_{\ast} (\UnitBall[\infty]([0,1]^d)) = \beta_\ast^{\MonteCarlo} (\UnitBall[\infty]([0,1]^d))
    = \begin{cases}
        \frac{1}{d}
        \cdot \frac{\alpha}{\gamma^\ast  + \alpha}
        & \text{if } \LayerFunc^\ast< \infty \\
        0
        & \text{else}.
      \end{cases}
\end{equation}
We now describe a particularly simple algorithm that achieves this optimal rate.
For $m \in \mathbb{N}_{\geq 2^d}$, let $n := \lfloor m^{1/d}\rfloor - 1$ and let 
\[
  A_m^{\mathrm{pw-const}}(f)(x)
  := \sum_{i_1 = 0}^n \cdots \sum_{i_d=0}^n
       f \left(\frac{i_1}{n},\dots , \frac{i_d}{n}\right)
       \indicator_{\left(\frac{i_1}{n},\dots , \frac{i_d}{n}\right) + \left[0,\frac{1}{n}\right)^d}(x)
\]
be the piecewise constant interpolation of $f$ on a tensor product grid
with sidelength $\frac{1}{n}$, where for $\tau \in [0,1]^d$ the function
$\indicator_{\tau + \left[0,\frac{1}{n}\right)^d}$ denotes the indicator function
of the set $\tau + \left[0,\frac{1}{n}\right)^d\subset \R^d$.
Since $(n+1)^d \le m$ it clearly holds that 
\[
    A_m^{\mathrm{pw-const}} \in \Alg_m (\UnitBall[\infty]([0,1]^d)).
\]
We have the following theorem.

\begin{theorem}\label{thm:intro3}
  Under the above assumptions on $\LayerFunc,\CoeffFunc$,
  the piecewise constant interpolation operator $A_m^{\mathrm{pw-const}}$
  constitutes an optimal algorithm for reconstructing functions
  in $\UnitBall[\infty]([0,1]^d)$ from point samples:
  we have
  \begin{equation}
    \begin{split}
      \beta^{\deterministic}_{\ast} (\UnitBall[\infty]([0,1]^d))
      & = \beta^{MC}_{\ast} (\UnitBall[\infty]([0,1]^d)) \\
      & = \sup
          \Big\{
            \lambda \geq 0
            \,\,\colon\,\,
            \exists \, C > 0 \,\,
              \forall \, m \in \N: \\ & \hspace{2.5cm}
                 \sup_{f \in \UnitBall[\infty]([0,1]^d)}
                   \| A_m^{\mathrm{pw-const}}(f) - f \|_{L^\infty([0,1]^d)}
                 \leq C \cdot m^{-\lambda}
          \Big\} .  
    \end{split}
    \label{eq:optimal_algorithm}
  \end{equation}
\end{theorem}

\begin{proof}
    Equation~(\ref{eq:optimal_sampling_rates}) yields the best possible rate
    with respect to the number of samples.
    Without loss of generality we assume that $\LayerFunc^\ast< \infty$
    (otherwise there is nothing to show as the optimal rate is zero).
    We now show that all rates below the optimal rate $\frac{1}{d}
    \cdot \frac{\alpha}{\gamma^\ast  + \alpha}$ can be achieved by piecewise constant interpolation,
    which readily implies our theorem.
    To this end, let $\lambda \in \bigl(0,\frac{1}{d} \cdot \frac{\alpha}{\gamma^\ast  + \alpha}\bigr)$
    be arbitrary and let $\beta := d \cdot \lambda \in (0, \frac{\alpha}{\gamma^\ast + \alpha}) \subset (0,1)$.
    It is easy to check that $\alpha > \frac{\beta}{1-\beta} \gamma^\ast$.
    Hence, Theorem~\ref{thm:intro2} (applied with $p=\infty$) implies the embedding
    $\ApproxSpace[\infty] ([0,1]^d)\hookrightarrow C^{0,\beta}([0,1]^d)$,
    yielding the existence of a constant $C_1 > 0$ with 
    \[
        \UnitBall[\infty]([0,1]^d)\subset C_1\cdot U^\beta,
    \]
    where $U^\beta$ denotes the unit ball in $C^{0,\beta}([0,1]^d)$.
    Therefore, 
    \begin{equation}\label{eq:pwconstest}
      \sup_{f \in \UnitBall[\infty]([0,1]^d)}
        \| A_m^{\mathrm{pw-const}}(f) - f \|_{L^\infty([0,1]^d)}
      \leq \sup_{f \in C_1\cdot U^\beta}
             \| A_m^{\mathrm{pw-const}}(f) - f \|_{L^\infty([0,1]^d)}
     \leq C \cdot m^{-\lambda},
    \end{equation}
    where the last inequality follows from the elementary fact
    that piecewise constant interpolation on the uniform tensor-product grid
    $\{ \frac{0}{n}, \frac{1}{n}, ..., \frac{n}{n} \}^d \strut$ achieves $L^\infty$ error
    of $\mathcal{O}(n^{-\beta}) \subset \mathcal{O}(m^{-\beta / d}) \!=\! \mathcal{O}(m^{-\lambda})$
    for functions $f \in U^\beta$.
    Since $\lambda<\frac{1}{d}\cdot \frac{\alpha}{\gamma^\ast + \alpha}$ was arbitrary,
    Equation~(\ref{eq:pwconstest}) proves Equation~(\ref{eq:optimal_algorithm})
    and thus our theorem.
\end{proof}

Theorem~\ref{thm:intro3} proves that, provided that the reconstruction error
is measured in the $L^\infty$ norm, \emph{the optimal ``learning'' algorithm
for reconstructing functions that are well approximable by ReLU neural networks
is simply given by piecewise constant interpolation on a tensor product grid},
at least if no additional information about the target function is given.

We finally remark that we expect all our results to hold (in slightly modified form)
if general bounded Lipschitz domains $\Omega\subset \mathbb{R}^d$ are considered
in place of the unit cube $[0,1]^d$.

\subsection{Related Work}

The study of approximation properties of neural networks has a long history, see for example
the surveys \cite{devore2020neural,grohs2019deep,pinkus1999approximation,berner2021modern}
and the references therein.
Function spaces related to neural networks have been studied in several works.
For example, \cite{barron1993universal,weinan2019barron,CarageaBarronBoundary}
study so-called \emph{Barron-type spaces} that arise in a certain sense
as approximation spaces of shallow neural networks.
The paper \cite{beneventano2020high} develops a calculus on spaces of functions
that can be approximated by neural networks without curse of dimension.
Closest to our work are the papers \cite{NNApproximationSpaces,GrohsTheoryToPracticeGap}
which study functional analytic properties of neural network approximation spaces
similar to the ones studied in the present paper.
The paper \cite{NNApproximationSpaces} also considers embedding results
between such spaces with different constraints on their underlying network architecture. 
However, none of the mentioned works considered embedding theorems of Sobolev type
comparable to those in the present paper.
An important reason for this is that the spaces considered in \cite{NNApproximationSpaces}
do not restrict the magnitude of the weights of the approximating networks,
so that essentially $\CoeffFunc \equiv \infty$ in our terminology.
The resulting approximation spaces then contain functions of very low smoothness.

\subsection{Structure of the Paper}

The structure of this paper is as follows.
In Section~\ref{sec:NNApproximationSpaces} we provide a definition
of the neural network approximation spaces considered in this paper
and establish some of their basic properties.
In Section~\ref{sec:Embeddings} we prove our embedding theorems.

\section{Definition of neural network approximation spaces}%
\label{sec:NNApproximationSpaces}


In this section, we review the mathematical definition of neural networks,
then formally introduce the neural network approximation spaces $\ApproxSpace ([0,1]^d)$,
and finally define the quantities $\LayerFunc^\ast$ and $\gamma^\ast (\LayerFunc,\CoeffFunc)$
that will turn out to be decisive for characterizing whether the embeddings
$\ApproxSpace([0,1]^d) \hookrightarrow C^{0,\beta}([0,1]^d)$
or $\ApproxSpace([0,1]^d) \hookrightarrow C([0,1]^d)$ hold.

\subsection{The mathematical formalization of neural networks}
\label{sub:NeuralNetworks}

In our analysis, it will be helpful to distinguish between a neural network $\Phi$
as a set of weights and the associated function $R_\varrho \Phi$ computed by the network.
Thus, we say that a \emph{neural network} is a tuple
${\Phi = \big( (A_1,b_1), \dots, (A_L,b_L) \big)}$,
with $A_\ell \in \R^{N_\ell \times N_{\ell-1}}$ and $b_\ell \in \R^{N_\ell}$.
We then say that ${\Architecture(\Phi) := (N_0,\dots,N_L) \in \N^{L+1}}$
is the \emph{architecture} of $\Phi$, $L(\Phi) := L$ is the \emph{number of layers}%
\footnote{Note that the number of \emph{hidden} layers is given by $H = L-1$.}
of $\Phi$, and ${W(\Phi) := \sum_{j=1}^L (\| A_j \|_{\ell^0} + \| b_j \|_{\ell^0})}$
denotes the \emph{number of (non-zero) weights} of $\Phi$.
The notation $\| A \|_{\ell^0}$ used here denotes the number of non-zero entries
of a matrix (or vector) $A$.
Finally, we write $\din(\Phi) := N_0$ and $\dout(\Phi) := N_L$
for the \emph{input and output dimension} of $\Phi$, and we set
$\WeightSize{\Phi} := \max_{j = 1,\dots,L} \max \{ \| A_j \|_{\infty}, \| b_j \|_{\infty} \}$,
where ${\| A \|_{\infty} := \max_{i,j} |A_{i,j}|}$.

To define the function $R_\varrho \Phi$ computed by $\Phi$, we need to specify an
\emph{activation function}.
In this paper, we will only consider the so-called \emph{rectified linear unit (ReLU)}
${\varrho : \R \to \R, x \mapsto \max \{ 0, x \}}$,
which we understand to act componentwise on $\R^n$, i.e.,
$\varrho \bigl( (x_1,\dots,x_n)\bigr) = \bigl(\varrho(x_1),\dots,\varrho(x_n)\bigr)$.
The function $R_\varrho \Phi : \R^{N_0} \to \R^{N_L}$ computed by the network $\Phi$
(its \emph{realization}) is then given by
\[
  R_\varrho \Phi
  := T_L \circ (\varrho \circ T_{L-1}) \circ \cdots \circ (\varrho \circ T_1)
  \quad \text{where} \quad
  T_\ell \, x = A_\ell \, x + b_\ell .
\]

\subsection{Neural network approximation spaces}
\label{sub:NeuralNetworkApproxSpaces}

\emph{Approximation spaces} \cite{DeVoreConstructiveApproximation} classify functions
according to how well they can be approximated by a family
$\boldsymbol{\Sigma} = (\Sigma_n)_{n \in \N}$ of certain ``simple functions''
of increasing complexity $n$, as $n \to \infty$.
Common examples consider the case where $\Sigma_n$ is the set of polynomials of degree $n$,
or the set of all linear combinations of $n$ wavelets.
The notion of \emph{neural network approximation spaces} was originally introduced in
\cite{NNApproximationSpaces}, where $\Sigma_n$ was taken to be a family of neural networks
of increasing complexity.
However, \emph{\cite{NNApproximationSpaces} does not impose any restrictions on the size
of the individual network weights}, which plays an important role in practice and
which is important in determining the regularity of the function implemented by the network.
Since the function spaces introduced in \cite{NNApproximationSpaces} ignore this quantity,
the spaces defined there \emph{never} embed into the Hölder-space $C^{0,\beta}$
for any $\beta > 0$.

For this reason, and following our previous work \cite{GrohsTheoryToPracticeGap},
we introduce a modified notion of neural network approximation spaces
that also takes the size of the individual network weights into account.
Precisely, given an input dimension $d \in \N$ (which we will keep fixed throughout this paper) and
non-decreasing functions ${\LayerFunc : \N \to \N_{\geq 2} \cup \{ \infty \}}$
and $\CoeffFunc : \N \to \N \cup \{ \infty \}$ (called the \textbf{depth-growth function}
and the \textbf{coefficient growth function}, respectively), we define
\[
  \NNSigma
  := \Big\{
       R_\varrho \Phi
       \,\, \colon
       \begin{array}{l}
         \Phi \text{ NN with }
         d_{\mathrm{in}}(\Phi) = d,
         d_{\mathrm{out}}(\Phi) = 1, \\
         W(\Phi) \leq n,
         L(\Phi) \leq \LayerFunc(n),
         \WeightSize{\Phi} \leq \CoeffFunc(n)
       \end{array}
     \Big\} .
\]
Then, given a measurable subset $\Omega \subset \R^d$ as well as $p \in (0,\infty]$
and $\alpha \in (0,\infty)$, for each measurable $f : \Omega \to \R$, we define
\[
  \PreNorm (f)
  := \max
     \Big\{
       \| f \|_{L^p(\Omega)} , \quad
       \sup_{n \in \N}
       \big[
         n^{\alpha}
         \cdot d_{p} \bigl(f, \NNSigma\bigr)
       \big]
     \Big\}
  \in [0,\infty] ,
\]
where $d_{p}(f, \Sigma) := \inf_{g \in \Sigma} \| f - g \|_{L^p(\Omega)}$.

The remaining issue is that since the set $\NNSigma$ is in general neither
closed under addition nor under multiplication with scalars, $\PreNorm$
is \emph{not a (quasi)-norm}.
To resolve this issue,
taking inspiration from the theory of Orlicz spaces
(see e.g.~\cite[Theorem~3 in Section~3.2]{RaoRenOrliczSpaces}),
we define the \emph{neural network approximation space quasi-norm}
$\| \cdot \|_{\ApproxSpace}$ as
\[
  \| f \|_{\ApproxSpace}
  := \inf \bigl\{ \theta > 0 \,\,\colon\,\, \PreNorm(f / \theta) \leq 1 \bigr\}
  \in [0,\infty],
\]
giving rise to the \emph{approximation space}
\[
  \ApproxSpace
  := \ApproxSpace (\Omega)
  := \bigl\{
       f \in L^p(\Omega)
       \,\,\colon\,\,
       \| f \|_{\ApproxSpace} < \infty
     \bigr\}
  .
\]
The following lemma summarizes the main elementary properties of these spaces.

\begin{lemma}\label{lem:ApproximationSpaceProperties}
  Let $\emptyset \neq \Omega \subset \R^d$ be measurable,
  let $p \in (0,\infty]$ and $\alpha \in (0,\infty)$.
  Then, the approximation space $\ApproxSpace := \ApproxSpace (\Omega)$ satisfies the following properties:
  \begin{enumerate}
    \item $(\ApproxSpace, \| \cdot \|_{\ApproxSpace})$ is a quasi-normed space.
          Precisely, given arbitrary measurable functions $f,g : \Omega \to \R$, it holds that
          $\| f + g \|_{\ApproxSpace} \leq C \cdot (\| f \|_{\ApproxSpace} + \| g \|_{\ApproxSpace})$
          for $C = C(\alpha,p)$.

    \item We have $\PreNorm (c f) \leq |c| \, \PreNorm(f)$ for $c \in [-1,1]$.

    \item $\PreNorm(f) \leq 1$ if and only if $\| f \|_{\ApproxSpace} \leq 1$.

    \item $\PreNorm(f) < \infty$ if and only if $\| f \|_{\ApproxSpace} < \infty$.

    \item $\ApproxSpace (\Omega) \hookrightarrow L^p(\Omega)$.
  \end{enumerate}
\end{lemma}

\begin{proof}
  For $p\in [1,\infty]$ this is precisely \cite[Lemma~2.1]{GrohsTheoryToPracticeGap}.
  The case $p\in (0,1)$ can be proven in a completely analogous way and is left to the reader.
\end{proof}

\begin{remark}
  While we introduced the spaces $\ApproxSpace(\Omega)$ for general domains $\Omega$,
  in what follows we will specialize to $\Omega = [0,1]^d$.
  We expect, however, that all our results will remain true
  (up to minor modifications) for general bounded Lipschitz domains.
\end{remark}

\subsection{Quantities characterizing the complexity of the network architecture}
\label{sub:GammaDefinition}

To conveniently summarize those aspects of the growth behavior of the functions
$\LayerFunc$ and $\CoeffFunc$ most relevant to us, we introduce three quantities
that will turn out to be crucial for characterizing whether the embeddings
$\ApproxSpace([0,1]^d) \hookrightarrow C^{0,\beta}([0,1]^d)$
or $\ApproxSpace([0,1]^d) \hookrightarrow C([0,1]^d)$ hold.
First, we set
\begin{equation}
  \LayerFunc^\ast
  := \sup_{n\in\N} \LayerFunc(n)
  \in \N_{\geq 2} \cup \{ \infty \} .
  \label{eq:MaximalDepth}
\end{equation}
Furthermore, we define
$\gamma^\ast (\LayerFunc, \CoeffFunc), \gamma^{\diamondsuit} (\LayerFunc, \CoeffFunc) \in (0,\infty]$
by
\begin{equation}
  \begin{split}
    \gamma^\ast (\LayerFunc, \CoeffFunc)
    & := \sup
         \bigg\{
           \gamma \in [0,\infty)
           \quad\colon\quad
           \exists \, L \in \N_{\leq \LayerFunc^\ast} : \,\,
             \limsup_{n \to \infty}
               \frac{(\CoeffFunc(n))^L \cdot n^{\lfloor L/2 \rfloor}}{n^\gamma}
               \in (0,\infty]
         \bigg\}, \\
    \gamma^{\diamondsuit} (\LayerFunc, \CoeffFunc)
    & := \inf
         \Big\{
           \gamma \in [0,\infty)
           \quad \colon \quad
           \exists \, C > 0 \,
             \forall \, n \in \N, L \in \N_{\leq \LayerFunc^\ast} : \,\,
               (\CoeffFunc(n))^L \cdot n^{\lfloor L/2 \rfloor}
               \leq C \cdot n^\gamma
         \Big\} ,
  \end{split}
  \label{eq:GammaDefinition}
\end{equation}
where we use the convention that $\sup \emptyset = 0$ and $\inf \emptyset = \infty$.
As we now show, the quantities $\gamma^\ast$ and $\gamma^{\diamondsuit}$ actually coincide.

\begin{lemma}\label{lem:GammaProperties}
  We have $\gamma^\ast (\LayerFunc,\CoeffFunc) = \gamma^{\diamondsuit}(\LayerFunc,\CoeffFunc)$.
  Furthermore, if $\LayerFunc^\ast = \infty$, then
  $\gamma^\ast (\LayerFunc,\CoeffFunc) = \gamma^{\diamondsuit}(\LayerFunc,\CoeffFunc) = \infty$.
\end{lemma}

\begin{proof}
  For brevity, write $\gamma^\ast := \gamma^\ast (\LayerFunc,\CoeffFunc)$
  and $\gamma^{\diamondsuit} := \gamma^{\diamondsuit} (\LayerFunc,\CoeffFunc)$.

  \medskip{}

  \textbf{Case~1 ($\LayerFunc^\ast = \infty$):}
  In this case, given any $\gamma > 0$,
  one can choose an even $L \in \N = \N_{\leq \LayerFunc^\ast}$
  satisfying $L > 2 \gamma$.
  It is then easy to see
  \(
    \limsup_{n \to \infty}
      \frac{(\CoeffFunc(n))^L \cdot n^{\lfloor L/2 \rfloor}}{n^\gamma}
    \geq \limsup_{n \to \infty}
           1
    = 1 \in (0,\infty]
    .
  \)
  By definition of $\gamma^\ast$, this implies $\gamma^\ast \geq \gamma$ for each
  $\gamma > 0$, and hence $\gamma^\ast = \infty$.

  \smallskip{}

  Next, we show that $\gamma^{\diamondsuit} = \infty$ as well.
  To see this, assume towards a contradiction that $\gamma^{\diamondsuit} < \infty$.
  By definition of $\gamma^{\diamondsuit}$, this means that there exist $\gamma \geq 0$
  and $C > 0$ satisfying
  \[
    C \cdot n^\gamma
    \geq (\CoeffFunc(n))^L \cdot n^{\lfloor L/2 \rfloor}
    \geq n^{\lfloor L/2 \rfloor}
    \qquad \forall \, n,L \in \N .
  \]
  But choosing an even $L \in \N$ satisfying $L > 2 \gamma$,
  we have $\lfloor L/2 \rfloor = L/2 > \gamma$, showing that the preceding displayed inequality
  cannot hold.
  This is the desired contradiction.

  \medskip{}

  \textbf{Case~2 ($\LayerFunc^\ast < \infty$):}
  We first show $\gamma^\ast \leq \gamma^{\diamondsuit}$.
  In case of $\gamma^{\diamondsuit} = \infty$, this is trivial, so that we can assume
  $\gamma^{\diamondsuit} < \infty$.
  Let $\gamma_0 \in [0,\infty)$ such that there exists $C > 0$ satisfying
  $(\CoeffFunc(n))^L \cdot n^{\lfloor L/2 \rfloor} \leq C \cdot n^{\gamma_0}$ for all
  $n \in \N$ and $L \in \N_{\leq \LayerFunc^\ast}$.
  Then, given any $L \in \N_{\leq \LayerFunc^\ast}$ and $\gamma > \gamma_0$, we have
  \[
    \limsup_{n \to \infty}
      \frac{(\CoeffFunc(n))^L \cdot n^{\lfloor L/2 \rfloor}}{n^{\gamma}}
    \leq \limsup_{n \to \infty}
           \frac{C \cdot n^{\gamma_0}}{n^{\gamma}}
    = 0.
  \]
  This easily implies $\gamma^\ast \leq \gamma_0$.
  Since this holds for all $\gamma_0$ as above, we see by definition of $\gamma^{\diamondsuit}$
  that $\gamma^\ast \leq \gamma^{\diamondsuit}$.

  We finally show that $\gamma^{\diamondsuit} \leq \gamma^\ast$.
  In case of $\gamma^\ast = \infty$, this is clear; hence, we can assume that $\gamma^\ast < \infty$.
  Let $\gamma > \gamma^\ast$ be arbitrary.
  Then
  \[
    0
    \leq \liminf_{n \to \infty}
           \frac{(\CoeffFunc(n))^{\LayerFunc^\ast} \cdot n^{\lfloor \LayerFunc^\ast/2 \rfloor}}{n^\gamma}
    \leq \limsup_{n \to \infty}
           \frac{(\CoeffFunc(n))^{\LayerFunc^\ast} \cdot n^{\lfloor \LayerFunc^\ast/2 \rfloor}}{n^\gamma}
    =    0 .
  \]
  Thus, since convergent sequences are bounded, we see that
  \(
    (\CoeffFunc(n))^{\LayerFunc^\ast} \cdot n^{\lfloor \LayerFunc^\ast / 2 \rfloor}
    \leq C \cdot n^\gamma
  \)
  for all $n \in \N$ and a suitable $C > 0$.
  Given any $L \in \N_{\leq \LayerFunc^\ast}$, we thus have
  \(
    (\CoeffFunc(n))^L \cdot n^{\lfloor L/2 \rfloor}
    \leq (\CoeffFunc(n))^{\LayerFunc^\ast} \cdot n^{\lfloor \LayerFunc^\ast / 2 \rfloor}
    \leq C \cdot n^\gamma .
  \)
  By definition of $\gamma^{\diamondsuit}$, this implies $\gamma^{\diamondsuit} \leq \gamma$.
  Since this holds for all $\gamma > \gamma^\ast$,
  we conclude $\gamma^{\diamondsuit} \leq \gamma^\ast$.
\end{proof}


\section{Sobolev-type embeddings for neural network approximation spaces}%
\label{sec:Embeddings}


In this section, we establish a characterization of the embeddings
$\ApproxSpace([0,1]^d) \hookrightarrow C([0,1]^d)$
and $\ApproxSpace([0,1]^d) \hookrightarrow C^{0,\beta}([0,1]^d)$
solely in terms of $\gamma^\ast(\LayerFunc,\CoeffFunc)$
and the quantities $d \in \N$, $\beta \in (0,1)$, and $p \in (0,\infty]$.
This is roughly similar to Sobolev embeddings,
which allow one to ``trade smoothness for increased integrability.''

\subsection{Sufficient Conditions for Embedding into \texorpdfstring{$C([0,1]^d)$}{C([0,1]ᵈ)}}

In this subsection we establish sufficient conditions for an embedding
$\ApproxSpace([0,1]^d) \hookrightarrow C([0,1]^d)$ to hold.
The following lemma---showing that functions with large (but finite) Lipschitz constant
but small $L^p$ norm have a controlled $L^\infty$ norm---is an essential ingredient for the proof.

We first make precise the notion of Lipschitz constant that we use in what follows.
Given $d\in \mathbb{N}$ and a function $f:[0,1]^d \to \mathbb{R}$ we define 
\[
  \Lip(f)
  := \sup_{x,y \in [0,1]^d, x \neq y}
       \frac{|f(x) - f(y)|}{\| x-y \|_{\ell^\infty}}
  \in [0,\infty]
  .
\]

\begin{lemma}\label{lem:LipschitzLpVersusLInfty}
  Let $p \in (0,\infty)$ and $d \in \N$.
  Then there exists a constant $C = C(d,p) > 0$ such that for each $0 < T \leq 1$
  and each $f \in L^p([0,1]^d) \cap C^{0,1}([0,1]^d)$, we have
  \[
    \| f \|_{L^\infty}
    \leq C \cdot \big( T^{-d/p} \, \| f \|_{L^p} + T \cdot \Lip(f) \big).
  \]
\end{lemma}

\begin{proof}
  For brevity, set $Q := [0,1]^d$.
  Let $f \in L^p([0,1]^d) \cap C^{0,1}([0,1]^d)$.
  Since $Q$ is compact and $f$ is continuous, there exists $x_0 \in Q$
  satisfying $|f(x_0)| = \| f \|_{L^\infty}$.
  Set $P := Q \cap (x_0 + [-T,T]^d)$.
  For $x \in P$, we have
  \[
    \| f \|_{L^\infty}
    = |f(x_0)|
    \leq |f(x_0) - f(x)| + |f(x)|
    \leq \Lip(f) \cdot \| x - x_0 \|_{\ell^\infty} + |f(x)|
    \leq \Lip(f) \cdot T + |f(x)| .
  \]
  Now, denoting the Lebesgue measure by $\LebesgueMeasure$,
  taking the $L^p(P)$ norm of the preceding estimate,
  and recalling that $L^p$ is a (quasi)-Banach space,
  we obtain a constant $C_1 = C_1(p) > 0$ satisfying
  \begin{align*}
    \| f \|_{L^\infty} \cdot [\LebesgueMeasure(P)]^{1/p}
    & =    \big\|
             \| f \|_{L^\infty}
           \big\|_{L^p(P)}
      \leq \big\|
             \Lip(f) \cdot T + |f|
           \big\|_{L^p(P)} \\
    & \leq C_1 \cdot \big(
                       \Lip(f) \cdot T \cdot [\LebesgueMeasure(P)]^{1/p}
                       + \| f \|_{L^p(P)}
                     \big) \\
    & \leq C_1 \cdot \big(
                       \Lip(f) \cdot T \cdot [\LebesgueMeasure(P)]^{1/p}
                       + \| f \|_{L^p}
                     \big) .
  \end{align*}
  Next, \cite[Lemma~A.2]{GrohsTheoryToPracticeGap}
  shows that $\LebesgueMeasure(P) \geq 2^{-d} \, T^d$, so that we finally see
  \[
    \| f \|_{L^\infty}
    \leq C_1 \cdot \bigl(\Lip(f) \cdot T + [\LebesgueMeasure(P)]^{-1/p} \, \| f \|_{L^p}\bigr)
    \leq C_1 \cdot \bigl(\Lip(f) \cdot T + 2^{d/p} \, T^{-d/p} \, \| f \|_{L^p}\bigr) ,
  \]
  which easily implies the claim.
\end{proof}

In order to derive a sufficient condition for the embedding
$\ApproxSpace([0,1]^d) \hookrightarrow C([0,1]^d)$, we will use the following
lemma that provides a bound for the Lipschitz constant of a function
implemented by a neural network, in terms of the size (number of non-zero weights)
of the network and of the magnitude of the individual network weights.

\begin{lemma}\label{lem:NetworkLipschitzEstimate}(\cite[Lemma~4.1]{GrohsTheoryToPracticeGap})
  Let $\LayerFunc : \N \to \N_{\geq 2} \cup \{ \infty \}$ and $\CoeffFunc : \N \to [1,\infty]$.
  Let $n \in \N$ and assume that $L := \LayerFunc(n)$ and $C := \CoeffFunc(n)$ are finite.
  Then each $F \in \NNSigma$ satisfies
  \[
    \Lip (F)
    \leq d \cdot C^L \cdot n^{\lfloor L/2 \rfloor} .
  \]
\end{lemma}

Now, we can derive a sufficient condition for the embedding
$\ApproxSpace([0,1]^d) \hookrightarrow C([0,1]^d)$.
Below, in Section~\ref{sec:necessary}, we will see that this sufficient condition
is essentially optimal.

\begin{theorem}\label{thm:SobolevTypeEmbeddingSufficient}
  Let $d \in \N$ and $p \in (0,\infty)$,
  and suppose that $\gamma^{\ast}(\LayerFunc,\CoeffFunc) < \infty$.
  If $\alpha > \frac{d}{p} \cdot \gamma^{\ast}(\LayerFunc,\CoeffFunc)$,
  then $\ApproxSpace ([0,1]^d) \hookrightarrow C([0,1]^d)$.
  Furthermore, for each
  $\gamma \in \bigl( \gamma^{\ast}(\LayerFunc,\CoeffFunc), \frac{p}{d} \, \alpha \bigr)$,
  we have the embedding
  \[
    \ApproxSpace([0,1]^d)
    \hookrightarrow \ApproxSpace[\infty][\frac{p}{d+p}(\alpha - \frac{d}{p}\gamma)] ([0,1]^d).
  \]
\end{theorem}

\begin{proof}
  As shown in \Cref{lem:GammaProperties}, the assumption
  $\gamma^{\ast}(\LayerFunc,\CoeffFunc) < \infty$
  ensures that $L := \LayerFunc^\ast < \infty$.
  Furthermore, since
  \(
    \gamma^{\diamondsuit}(\LayerFunc,\CoeffFunc)
    = \gamma^{\ast}(\LayerFunc,\CoeffFunc)
    < \frac{p}{d} \alpha
  \)
  we see by definition of $\gamma^{\diamondsuit}$ (see \Cref{eq:GammaDefinition})
  that there exist $\gamma < \frac{p}{d} \alpha$
  (arbitrarily close to $\gamma^{\ast}(\LayerFunc,\CoeffFunc)$)
  and $C = C(\gamma,\LayerFunc,\CoeffFunc) > 0$ satisfying
  $(\CoeffFunc(n))^{L} \cdot n^{\lfloor L/2 \rfloor} \leq C \cdot n^{\gamma}$ for all $n \in \N$.

  Let $f \in \ApproxSpace([0,1]^d)$ with $\| f \|_{\ApproxSpace} \leq 1$.
  By \Cref{lem:ApproximationSpaceProperties}, this implies $\PreNorm(f) \leq 1$,
  so that for each $m \in \N$, we can choose $F_m \in \NNSigma[2^m]$ satisfying
  $\| f - F_m \|_{L^p} \leq 2 \cdot 2^{-\alpha m}$.
  Since furthermore $\| f \|_{L^p} \leq 1$, this remains true for $m = 0$ if we set $F_0 := 0$.
  Note that there exists a constant $C_1 = C_1(p) > 0$ satisfying
  \[
    \| F_{m+1} - F_m \|_{L^p}
    \leq C_1 \cdot \big( \| F_{m+1} - f \|_{L^p} + \| f - F_m \|_{L^p} \big)
    \leq C_1 \cdot \big( 2 \cdot 2^{-\alpha(m+1)} + 2 \cdot 2^{-\alpha m} \big)
    \leq 4 C_1 \cdot 2^{- \alpha m} .
  \]
  Furthermore, note as a consequence of \Cref{lem:NetworkLipschitzEstimate} and because of
  $\LayerFunc(n) \leq L$ for all $n \in \N$ that
  \begin{align*}
    \Lip(F_{m+1} - F_m)
    & \leq \Lip(F_{m+1}) + \Lip(F_m)
      \leq d \cdot \bigl(
                     [\CoeffFunc(2^{m+1})]^{L} \, (2^{m+1})^{\lfloor L/2 \rfloor}
                     + [\CoeffFunc(2^{m})]^{L} \, (2^{m})^{\lfloor L/2 \rfloor}
                   \bigr) \\
    & \leq d \, C \cdot \bigl( (2^{m+1})^\gamma + (2^m)^\gamma \bigr)
      \leq C_2 \cdot 2^{\gamma m}
  \end{align*}
  for a suitable constant $C_2 = C_2(d,\gamma,\LayerFunc,\CoeffFunc) > 0$.

  Now, set $\theta := \frac{p}{p+d} (\gamma + \alpha) > 0$.
  Then, for each $m \in \N$, we can apply \Cref{lem:LipschitzLpVersusLInfty}
  with $T = T_m := 2^{- m \theta} \in (0,1]$ to deduce
  for a suitable constant $C_3 = C_3(d,p) > 0$ that
  \begin{align*}
    \| F_{m+1} - F_m \|_{L^\infty}
    & \leq C_3 \cdot \bigl(
                       T_m^{-d/p} \, \| F_{m+1} - F_m \|_{L^p} + T_m \cdot \Lip(F_{m+1} - F_m)
                     \bigr) \\
    & \leq C_3 \cdot \bigl(
                       4 C_1 \cdot 2^{\frac{d}{p} \theta m} \cdot 2^{-\alpha m}
                       + C_2 \cdot 2^{-\theta m} \cdot 2^{\gamma m}
                     \bigr) \\
    & \leq C_4 \cdot 2^{- \mu m} ,
  \end{align*}
  for a suitable constant $C_4 = C_4(d,p,\gamma,\LayerFunc,\CoeffFunc) > 0$
  and $\mu := \frac{p}{d+p} (\alpha - \frac{d}{p}\gamma) > 0$.
  Therefore, we see for $M \geq m \geq m_0$ that
  \begin{equation}
    \| F_M - F_m \|_{L^\infty}
    \leq \sum_{\ell=m}^{M-1}
           \| F_{\ell+1} - F_\ell \|_{L^\infty}
    \leq C_4 \sum_{\ell=m_0}^\infty 2^{- \mu \ell}
    \xrightarrow[m_0\to\infty]{} 0 ,
    \label{eq:SobolevTypeEmbeddingSufficient1}
  \end{equation}
  showing that $(F_m)_{m \in \N} \subset C([0,1]^d)$ is a Cauchy sequence.
  Since $F_m \to f$ in $L^p$, this implies (after possibly redefining $f$ on a null-set)
  that $f \in C([0,1]^d)$ and $F_m \to f$ uniformly.
  Finally, we see
  \begin{equation}
    \| f \|_{L^\infty}
    \!=\!\! \lim_{m \to \infty} \| F_m \|_{L^\infty}
    \!\leq\!\! \lim_{m \to \infty} \sum_{\ell=0}^{m-1} \| F_{\ell+1} - F_\ell \|_{L^\infty}
    \!\leq C_4 \cdot \sum_{\ell=0}^\infty 2^{- \mu \ell}
    \leq C_5 = C_5 (d,p,\gamma,\alpha,\LayerFunc,\CoeffFunc) .
    \label{eq:SobolevTypeEmbeddingSufficient2}
  \end{equation}
  Since this holds for all $f \in \ApproxSpace ([0,1]^d)$ with $\| f \|_{\ApproxSpace} \leq 1$,
  we have shown $\ApproxSpace ([0,1]^d) \hookrightarrow C([0,1]^d)$.

  \medskip{}

  We complete the proof by showing
  $\ApproxSpace ([0,1]^d) \hookrightarrow \ApproxSpace[\infty][\mu]([0,1]^d)$.
  By definition of $\mu$ and since $\gamma < \frac{p}{d}\alpha$ can be chosen arbitrarily
  close to $\gamma^{\ast}(\LayerFunc,\CoeffFunc)$, this implies the second claim of the theorem.
  To see $\ApproxSpace([0,1]^d) \hookrightarrow \ApproxSpace[\infty][\mu]([0,1]^d)$,
  note thanks to \Cref{eq:SobolevTypeEmbeddingSufficient1} and because of $F_m \to f$ uniformly that
  \[
    \| f - F_m \|_{L^\infty}
    = \lim_{M\to\infty} \| F_M - F_m \|_{L^\infty}
    \leq C_4 \sum_{\ell=m}^\infty 2^{-\mu \ell}
    \leq \frac{C_4}{1 - 2^{-\mu}} 2^{- \mu m}
    =:   C_6 \cdot 2^{- \mu m} ,
  \]
  for a suitable constant $C_6 = C_6(d,p,\gamma,\alpha,\LayerFunc,\CoeffFunc)$.

  Now, given any $t \in \N_{\geq 2}$, choose $m \in \N$ with $2^m \leq t \leq 2^{m+1}$.
  Since $F_m \in \NNSigma[2^m] \subset \NNSigma[t]$, we then see
  \(
    t^\mu \, d_\infty (f, \NNSigma[t])
    \leq t^\mu \, \| f - F_m \|_{L^\infty}
    \leq 2^{\mu (m+1)} C_6 \cdot 2^{- \mu m}
    =    2^\mu C_6 .
  \)
  For $t = 1$, we have $t^\mu \, d_\infty(f, \NNSigma[t]) \leq \| f \|_{L^\infty} \leq C_5$;
  see \Cref{eq:SobolevTypeEmbeddingSufficient2}.
  Overall, setting $C_7 := \max \{ 1, 2^\mu C_6, C_5 \}$ and recalling
  \Cref{lem:ApproximationSpaceProperties}, we see $\PreNorm[\infty]<\mu>(C_7^{-1} f) \leq 1$
  and hence $\| C_7^{-1} f \|_{\ApproxSpace[\infty][\mu]} \leq 1$, meaning
  $\| f \|_{\ApproxSpace[\infty][\mu]} \leq C_7$.
  All in all, this shows that indeed
  $\ApproxSpace([0,1]^d) \hookrightarrow \ApproxSpace[\infty][\mu] ([0,1]^d)$.
\end{proof}

\subsection{Sufficient Conditions for Embedding into \texorpdfstring{$C^{0,\beta}([0,1]^d)$}{Hölder spaces}}

Our next goal is to prove a sufficient condition for the embedding
$\ApproxSpace([0,1]^d) \hookrightarrow C^{0,\beta}([0,1]^d)$.
We first introduce some notation. 
For $d\in \mathbb{N}$, $\beta\in (0,1]$ and $f:[0,1]^d\to \mathbb{R}$ we define
\[
  \Lip_\beta (f) := \sup_{x,y \in [0,1]^d, x \neq y}
                      \frac{|f(x) - f(y)|}{\|x - y\|_{\ell^\infty}^\beta}
                 \in [0,\infty].
\]
Then, the Hölder space $C^{0,\beta}([0,1]^d)$ consists of all (necessarily continuous)
functions $f : [0,1]^d \to \R$ satisfying
\[
  \| f \|_{C^{0,\beta}}
  := \| f \|_{L^\infty} + \Lip_\beta(f)
  <  \infty .
\]
It is well-known that $C^{0,\beta}([0,1]^d)$ is a Banach space with this norm.

The proof of our sufficient embedding condition
relies on the following analogue of \Cref{lem:LipschitzLpVersusLInfty}.

\begin{lemma}\label{lem:HoelderConstantBound}
  Let $p \in (0,\infty]$, $\beta \in (0,1)$, and $d \in \N$.
  Then there exists a constant $C = C(d,p,\beta) \!>\! 0$
  such that for each $0 < T \leq 1$ and each $f \in L^p([0,1]^d) \cap C^{0,1}([0,1]^d)$, we have
  \[
    \Lip_\beta (f)
    \leq C \cdot \big( T^{-\frac{d}{p} - \beta} \| f \|_{L^p} + T^{1 - \beta} \Lip(f) \big).
  \]
\end{lemma}

\begin{proof}
  For $x,y \in [0,1]^d$, we have
  \(
    |f(x) - f(y)|^{1-\beta}
    \leq \big( 2 \, \| f \|_{L^\infty} \big)^{1-\beta}
  \)
  and hence
  \[
    |f(x) - f(y)|
    = |f(x) - f(y)|^\beta \cdot |f(x) - f(y)|^{1-\beta}
    \leq 2^{1-\beta} \,
         \| f \|_{L^\infty}^{1-\beta}
         \cdot \big( \Lip(f) \cdot \|x - y\|_{\ell^\infty} \big)^\beta ,
  \]
  and this easily implies
  $\Lip_\beta(f) \leq 2^{1-\beta} \cdot \| f \|_{L^\infty}^{1-\beta} \cdot (\Lip(f))^\beta$.

  Let us first consider the case $p < \infty$.
  In this case, we combine the preceding bound with the estimate of $\| f \|_{L^\infty}$
  provided by \Cref{lem:LipschitzLpVersusLInfty} to obtain a constant $C_1 = C_1(d,p) > 0$
  satisfying
  \begin{align*}
    \Lip_\beta(f)
    & \leq 2^{1-\beta} C_1^{1-\beta}
           \cdot \bigl( T^{-d/p} \, \| f \|_{L^p} + T \cdot \Lip(f) \bigr)^{1-\beta}
           \cdot \big( \Lip(f) \big)^\beta \\
    & \leq 2^{1-\beta} C_1^{1-\beta}
           \cdot \big(
                   T^{-(1-\beta) d / p}
                   \cdot \| f \|_{L^p}^{1-\beta}
                   \cdot (\Lip(f))^\beta
                   + T^{1-\beta} \cdot \Lip(f)
                 \big)
           .
  \end{align*}
  The second step used the elementary estimate $(a + b)^\theta \leq a^\theta + b^\theta$
  for $\theta \in (0,1)$ and $a,b \geq 0$.

  The preceding estimate shows that it suffices to prove that
  \[
    T^{-(1-\beta) d / p}
    \cdot \| f \|_{L^p}^{1-\beta}
    \cdot (\Lip(f))^\beta
    \leq C_2
         \cdot \big(
                 T^{-\frac{d}{p} - \beta} \| f \|_{L^p}
                 + T^{1-\beta} \Lip(f)
               \big)
  \]
  for a suitable constant $C_2 = C_2(d,p,\beta) > 0$.
  To prove the latter estimate, note for $a, b > 0$ and $\theta \in (0,1)$ that
  \begin{equation}
    a^\theta \cdot b^{1-\theta}
    = \exp \bigl(\theta \ln a + (1-\theta) \ln b \bigr)
    \leq \theta \exp(\ln a) + (1-\theta) \exp(\ln b)
    =    \theta a + (1-\theta) b
    ,
    \label{eq:YoungTypeInequality}
  \end{equation}
  thanks to the convexity of $\exp$.
  The preceding estimate clearly remains valid if $a = 0$ or $b = 0$.
  Applying this with $\theta = \beta$, $a = T^{1-\beta} \Lip(f)$,
  and $b = T^{-\frac{d}{p}-\beta} \| f \|_{L^p}$, we thus obtain the needed estimate
  \begin{align*}
    T^{-(1-\beta) d / p}
    \cdot \| f \|_{L^p}^{1-\beta}
    \cdot (\Lip(f))^\beta
    & = \big( T^{1-\beta} \Lip(f) \big)^\beta
        \cdot \big( T^{-\frac{d}{p} - \beta} \, \| f \|_{L^p} \big)^{1-\beta} \\
    & \leq \beta \, T^{1-\beta} \Lip(f) + (1-\beta) \, T^{-\frac{d}{p} - \beta} \, \| f \|_{L^p} \\
    & \leq T^{1-\beta} \Lip(f) + T^{-\frac{d}{p} - \beta} \, \| f \|_{L^p}
    .
  \end{align*}
  This completes the proof in the case $p < \infty$.

  In case of $p = \infty$, we combine the estimate from the beginning of the proof
  with \Cref{eq:YoungTypeInequality} to obtain
  \begin{align*}
    \Lip_\beta(f)
    & \leq 2^{1-\beta} \cdot \| f \|_{L^\infty}^{1-\beta} \cdot (\Lip(f))^\beta
      =    2^{1-\beta}
           \cdot \bigl( T^{-\beta} \, \| f \|_{L^\infty} \bigr)^{1-\beta}
           \cdot \bigl(T^{1-\beta} \, \Lip(f)\bigr)^\beta \\
    & \leq 2^{1-\beta}
           \cdot \big(
                   (1-\beta) T^{-\beta} \| f \|_{L^\infty}
                   + \beta T^{1-\beta} \, \Lip(f)
                 \big) \\
    & \leq 2^{1-\beta}
           \cdot \big(
                   T^{-\frac{d}{p} - \beta} \| f \|_{L^p}
                   + T^{1-\beta} \Lip(f)
                 \big) ,
  \end{align*}
  where the last step used that $p = \infty$.
\end{proof}

We now prove our sufficient criterion for the embedding
$\ApproxSpace([0,1]^d) \hookrightarrow C^{0,\beta}([0,1]^d)$.

\begin{theorem}\label{thm:NNIntoHoelderSpace}
  Let $d \in \N$, $\beta \in (0,1)$, and $p \in (0,\infty]$,
  and suppose that $\gamma^\ast(\LayerFunc,\CoeffFunc) < \infty$.
  If
  \[
    \alpha
    > \frac{\beta + \frac{d}{p}}{1 - \beta}
      \cdot \gamma^\ast(\LayerFunc,\CoeffFunc)
    ,
  \]
  then $\ApproxSpace([0,1]^d) \hookrightarrow C^{0,\beta}([0,1]^d)$.
\end{theorem}

\begin{proof}
  The assumption on $\alpha$ implies in particular
  that $\alpha > \frac{d}{p} \gamma^\ast(\LayerFunc,\CoeffFunc)$.
  Hence, \Cref{thm:SobolevTypeEmbeddingSufficient} shows
  in case of $p < \infty$ that $\ApproxSpace([0,1]^d) \hookrightarrow C([0,1]^d)$,
  and this embedding trivially holds if $p = \infty$.
  Thus, it remains to show for $f \in \ApproxSpace([0,1]^d)$
  with $\| f \|_{\ApproxSpace} \leq 1$ that $\Lip_\beta (f) \lesssim 1$.

  Let $\nu := (\beta + \frac{d}{p}) / (1 - \beta)$.
  Recall from \Cref{lem:GammaProperties} and the assumptions of the present theorem that
  \(
    \gamma^{\diamondsuit}(\LayerFunc,\CoeffFunc)
    = \gamma^\ast(\LayerFunc,\CoeffFunc)
    < \alpha / \nu
    < \infty
  \)
  and hence that $L := \LayerFunc^\ast \in \N$.
  By definition of $\gamma^{\diamondsuit}(\LayerFunc,\CoeffFunc)$,
  we thus see that there exist $\gamma > 0$ and $C = C(\gamma,\LayerFunc,\CoeffFunc) > 0$
  satisfying $\gamma < \alpha / \nu$ and
  $(\CoeffFunc(n))^L \cdot n^{\lfloor L/2 \rfloor} \leq C \cdot n^{\gamma}$ for all $n \in \N$.

  Let $f \in \ApproxSpace([0,1]^d)$ with $\| f \|_{\ApproxSpace} \leq 1$.
  By \Cref{lem:ApproximationSpaceProperties}, this implies $\PreNorm(f) \leq 1$,
  meaning that for each $m \in \N$ we can choose $F_m \in \NNSigma[2^m]$
  satisfying $\| f - F_m \|_{L^p} \leq 2 \cdot 2^{-\alpha m}$.
  Since furthermore $\| f \|_{L^p} \leq 1$, this remains true for $m = 0$ if we set $F_0 := 0$.
  As in the proof of \Cref{thm:SobolevTypeEmbeddingSufficient}, we then obtain constants
  $C_1 = C_1(p) > 0$ and $C_2 = C_2(d,\gamma,\LayerFunc,\CoeffFunc) > 0$ satisfying
  \[
    \| F_{m+1} - F_m \|_{L^p} \leq 4 C_1 \cdot 2^{-\alpha m}
    \quad \text{and} \quad
    \Lip(F_{m+1} - F_m)
    \leq C_2 \cdot 2^{\gamma m}
    \quad \forall \, m \in \N .
  \]

  Now, set $\theta := (\alpha + \gamma)/(1 + \frac{d}{p})$.
  A direct computation shows that
  \[
    \theta \cdot \left(\frac{d}{p} + \beta\right) - \alpha
    = \gamma - \theta \cdot (1 - \beta)
    = \frac{1 - \beta}{1 + \frac{d}{p}} \cdot (\nu \gamma - \alpha)
    =: \mu
    < 0 .
  \]
  Thus, applying \Cref{lem:HoelderConstantBound} with $T = 2^{-\theta m}$, we obtain a constant
  $C_3 = C_3(d,p,\beta) > 0$ satisfying
  \begin{align*}
    \Lip_\beta (F_{m+1} - F_m)
    & \leq C_3
           \cdot \big(
                   T^{-\frac{d}{p} - \beta} \| F_{m+1} - F_m \|_{L^p}
                   + T^{1 - \beta} \, \Lip(F_{m+1} - F_m)
                 \big) \\
    & \leq C_3
           \cdot \big(
                   4 C_1 \cdot 2^{m \cdot (\theta \cdot (\frac{d}{p} + \beta) - \alpha)}
                   + C_2 \cdot 2^{m \cdot (\gamma - \theta \cdot (1 - \beta))}
                 \big) \\
    & =    C_3 \cdot (4C_1 + C_2) \cdot 2^{\mu m}
      =:   C_4 \cdot 2^{\mu m} .
  \end{align*}

  Next, note for $g \in C^{0,\beta}([0,1]^d)$ and $x,y \in [0,1]^d$ that
  \[
    |g(x)|
    \leq |g(x) - g(y)| + |g(y)|
    \leq \Lip_\beta(g) \cdot \|x - y\|_{\ell^\infty}^\beta + |g(y)|
    \leq \Lip_\beta (g) + |g(y)| .
  \]
  Taking the $L^p$-norm (with respect to $y$) of this estimate shows that
  \[
    |g(x)|
    = \big\| |g(x)| \big\|_{L^p([0,1]^d, d y)}
    \leq \big\| \Lip_\beta (g) + |g(y)| \big\|_{L^p([0,1]^d, d y)}
    \leq C_5 \cdot \bigl(\Lip_\beta (g) + \| g \|_{L^p}\bigr)
    \,\,\, \forall \, x \!\in\! [0,1]^d,
  \]
  for a suitable constant $C_5 = C_5(p) > 0$.
  Setting $\sigma := \max \{ \mu, -\alpha \} < 0$ and applying the preceding
  estimate to $g = F_{m+1} - F_m$, we see for a suitable constant $C_6 = C_6(d,p,\beta) > 0$
  that $\| F_{m+1} - F_m \|_{\sup} \leq C_6 \cdot 2^{\sigma m}$ for all $m \in \N_0$.
  Hence, the series $F := \sum_{m=0}^\infty (F_{m+1} - F_m)$ converges uniformly,
  so that we get $F \in C^{0,\beta}([0,1]^d)$ with
  \[
    \Lip_\beta(F)
    \leq \sum_{m=0}^\infty
           \Lip_\beta(F_{m+1} - F_m)
    \leq C_4 \cdot \sum_{m=0}^\infty
                     2^{\mu m}
    =    \frac{C_4}{1 - 2^\mu}
    =:   C_7
    <    \infty .
  \]
  Since $\sum_{m=0}^N (F_{m+1} - F_m) = F_{N+1} \to f$ with convergence in $L^p([0,1]^d)$
  as $N \to \infty$, we see that $F = f$ almost everywhere,
  and hence (after changing $f$ on a null-set) $f \in C^{0,\beta}([0,1]^d)$
  with $\Lip_\beta (f) \leq C_7$.
  As seen above, this completes the proof.
\end{proof}

\subsection{Necessary Conditions}\label{sec:necessary}

In this section we complement our sufficient conditions by corresponding necessary conditions.
Our main result reads as follows.

\begin{theorem}\label{thm:SobolevTypeEmbeddingNecessary}
  Let $p \in (0,\infty]$, $\beta \in (0,1]$, and $d \in \N$.
  Then:
  \begin{itemize}
    \item If $p < \infty$ and $\gamma^{\ast}(\LayerFunc,\CoeffFunc) > \alpha / (d / p)$,
          then the embedding $\ApproxSpace ([0,1]^d) \hookrightarrow C([0,1]^d)$ does \emph{not} hold.

    \item If $\gamma^\ast(\LayerFunc,\CoeffFunc) > \frac{1-\beta}{\beta + \frac{d}{p}} \alpha$,
          then the embedding $\ApproxSpace ([0,1]^d) \hookrightarrow C^{0,\beta}([0,1]^d)$
          does \emph{not} hold.
  \end{itemize}
\end{theorem}

\begin{rem*}
  A combination of the preceding result with \Cref{thm:SobolevTypeEmbeddingSufficient}
  provides a \emph{complete characterization} of the embedding
  $\ApproxSpace ([0,1]^d) \hookrightarrow C([0,1]^d)$, except in the case where
  $\gamma^{\ast}(\LayerFunc,\CoeffFunc) = \alpha / (d/p)$.
  This remaining case is more subtle and seems to depend on the precise growth
  of the function $\CoeffFunc$.

  Likewise, combining the preceding result with \Cref{thm:NNIntoHoelderSpace}
  provides a \emph{complete characterization} of the embedding
  $\ApproxSpace ([0,1]^d) \hookrightarrow C^{0,\beta}([0,1]^d)$, except in the ``critical'' case
  where $\gamma^{\ast}(\LayerFunc,\CoeffFunc) = \frac{1-\beta}{\beta + \frac{d}{p}} \, \alpha$.
\end{rem*}

\begin{proof}
  The proof is divided into several steps, where the first two steps
  are concerned with constructing and analyzing a certain auxiliary function $\zeta_{M'}$
  which is then used to construct counterexamples to the considered embeddings.

  \medskip{}

  \textbf{Step~1:}
  Given $M' \geq 0$, define
  \[
    \zeta_{M'} : \quad
    \R^d \to \R, \quad
    x \mapsto \varrho \left(1 - M' \, \sum_{i=1}^d x_i\right).
  \]
  In this step, we show for any $n \in \N$, $0 \leq C \leq \CoeffFunc(n)$,
  $L \in \N_{\geq 2}$ with $L \leq \LayerFunc(n)$ and any $M' \geq 1$
  and $0 \leq M \leq C^L \, n^{\lfloor L/2 \rfloor}$ that
  \begin{equation}
    \tfrac{M}{M'} \cdot \zeta_{M'}
    \in \NNSigma[(d+L)n] .
    \label{eq:SobolevTypeEmbeddingNecessaryMainStep}
  \end{equation}
  To prove this, set $\mu := C^L \, n^{\lfloor L/2 \rfloor}$
  and $D := C \cdot (1)\in \R^{1 \times 1}$, as well as
  \[
    b_1 := \frac{C}{M'} \cdot (1,\dots,1)^T \in \R^n,
    \quad
    A := C \cdot (1,\dots,1) \in \R^{1 \times n},
    \quad
    B := C \cdot (1,\dots,1)^T \in \R^{n \times 1},
  \]
  and finally
  \[
    A_1
    := - C \cdot \left(
                   \begin{matrix}
                     1      & \cdots & 1 \\
                     \vdots & \ddots & \vdots \\
                     1      & \cdots & 1
                   \end{matrix}
                 \right)
    \in \R^{n \times d} .
  \]
  Note because of $M' \geq 1$ that
  \(
    \| b_1 \|_{\infty},
    \| A \|_{\infty},
    \| B \|_{\infty},
    \| D \|_{\infty},
    \| A_1 \|_{\infty}
    \leq C
    \leq \CoeffFunc(n)
    \leq \CoeffFunc ( (L+d) n)
    .
  \)
  Now, we distinguish two cases.

  \emph{Case~1 ($L \geq 2$ is even):}
  In this case, define
  \[
    \Phi := \big(
              (A_1,b_1),
              (A,0),
              \underbrace{
                (B,0),
                (A,0),
                \dots,
                (B,0),
                (A,0)
              }_{\frac{L-2}{2} \text{ copies of } (B,0),(A,0)}
            \big) .
  \]
  Note that $W(\Phi) \leq d n + n + n + \frac{L-2}{2} (n + n) \leq (d+L) n$
  and $L(\Phi) = L \leq \LayerFunc(n) \leq \LayerFunc( (d+L) n)$,
  showing that $R_\varrho \Phi \in \NNSigma[(d+L) n]$.
  Because of $A (\varrho (B x)) = C \sum_{j=1}^n \varrho(C x) = C^2 \, n \cdot \varrho(x)$
  for $x \in \R$ and because of
  \[
    A \cdot \varrho (A_1 x + b_1)
    = C \cdot \sum_{j=1}^n
                \varrho\bigl(\tfrac{C}{M'} - {\textstyle\sum_{i=1}^d} C \, x_i\bigr)
    = \frac{C^2 \, n}{M'}
      \cdot \varrho\bigl(1 - M' \, {\textstyle\sum_{i=1}^d} x_i\bigr)
    = \frac{C^2 \, n}{M'} \cdot \zeta_{M'} (x)
  \]
  for $x \in \R^d$, it follows that
  \(
    R_\varrho \Phi
    = C^{L-2}
      \cdot n^{\frac{L-2}{2}}
      \cdot \frac{C^2 \, n}{M'}
      \cdot \zeta_{M'}
    = \frac{\mu}{M'} \cdot \zeta_{M'} .
  \)
  Since $c \NNSigma[(d+L)n] \subset \NNSigma[(d+L)n]$ for all $c \in [-1,1]$
  and because of $0 \leq M \leq \mu$, this proves \Cref{eq:SobolevTypeEmbeddingNecessaryMainStep}
  in this case.

  \medskip{}

  \emph{Case~2 ($L \geq 3$ is odd):}
  In this case, define
  \[
    \Phi
    := \big(
         (A_1,b_1),
         (A,0),
         \underbrace{
           (B,0),(A,0),\dots,(B,0),(A,0)
         }_{\frac{L-3}{2} \text{ copies of } (B,0),(A,0)} ,
         (D,0)
       \big) .
  \]
  Then $W(\Phi) \leq d n + n + n + \frac{L-3}{2}(n + n) + 1 \leq (d + L) n$
  and $L(\Phi) = L \leq \LayerFunc(n) \leq \LayerFunc( (d+L) n)$,
  so that $R_\varrho \Phi \in \NNSigma[(d+L) n]$.
  Using essentially the same arguments as in the previous case, we see
  $R_\varrho \Phi = \frac{\mu}{M'} \cdot \zeta_{M'}$,
  which then implies that \Cref{eq:SobolevTypeEmbeddingNecessaryMainStep} holds.

  \medskip{}

  \textbf{Step~2:} We show
  $\big\| \zeta_{M'} \big\|_{L^p([0,1]^d)} \leq (M')^{-d/p}$.
  To see this, note for $x \in [0,1]^d$ that $\sum_{i=1}^d x_i \geq 0$ and hence
  $1 - M' \, \sum_{i=1}^d x_i \leq 1$, showing that
  $0 \leq \zeta_{M'}(x) = \varrho\bigl(1 - M' \, \sum_{i=1}^d x_i\bigr) \leq \varrho(1) = 1$.
  Next, note that if $0 \neq \zeta_{M'}(x) = \varrho \big( 1 - M' \, \sum_{i=1}^d x_i \big)$,
  then $1 - M' \, \sum_{i=1}^d x_i > 0$, which is only possible if
  $x_j \leq \sum_{i=1}^d x_i \leq \frac{1}{M'}$ for all $j \in \{ 1,\dots,d \}$,
  that is, if $x \in \frac{1}{M'} [0,1]^d$.
  Therefore, we see as claimed that
  \[
    \big\| \zeta_{M'} \big\|_{L^p([0,1]^d)}
    \leq \big[ \LebesgueMeasure\bigl(\tfrac{1}{M'} [0,1]^d\bigr) \big]^{1/p}
    =    (M')^{-d/p} .
  \]

  \medskip{}

  \textbf{Step~3:} In this step, we prove the first part of the theorem.
  By the assumptions of that part,
  we have $p < \infty$ and $\gamma^\ast (\LayerFunc,\CoeffFunc) > \alpha / (d/p)$.
  By definition of $\gamma^\ast$, this implies that there exists $\gamma > \alpha / (d/p)$
  and some $L \in \N_{\leq \LayerFunc^\ast}$ satisfying
  \[
    c_0
    := \limsup_{n \to \infty}
       \Big[
         (\CoeffFunc(n))^L \cdot n^{\lfloor L/2 \rfloor - \gamma}
       \Big]
    \in (0,\infty].
  \]
  Since $\LayerFunc^\ast \geq 2$, we can clearly assume that $L \geq 2$.
  Choose a subsequence $(n_k)_{k \in \N}$ such that
  $c_0 = \lim_{k \to \infty} (\CoeffFunc(n_k))^L \cdot n_k^{\lfloor L/2 \rfloor - \gamma}$.
  Then, setting $c := \min \{ 1, c_0 / 2 \}$, there exists $k_0 \in \N$ satisfying
  $(\CoeffFunc(n_k))^L \cdot n_k^{\lfloor L/2 \rfloor} \geq c \cdot n_k^\gamma > 0$
  and $L \leq \LayerFunc(n_k)$ for all $k \in \N_{\geq k_0}$.

  Choose $\kappa \in (0,1]$ satisfying $\kappa \leq c$ and $\kappa \cdot (d + L)^\alpha \leq 1$.
  Furthermore, define
  \[
    \theta
    := \tfrac{1}{2} (\gamma + \tfrac{\alpha p}{d}) \in (\tfrac{\alpha p}{d}, \gamma)
    \text{ and }
    \beta := \theta + \min \{ \delta_1, \delta_2 \}
    \quad \text{ where } \quad
    \delta_1 := \tfrac{\theta d}{p} - \alpha > 0
    \text{ and }
    \delta_2 := \gamma - \theta > 0
  \]
  and choose $k_1 \in \N$ satisfying $n_k^\theta \leq \kappa \cdot n_k^\beta$ for all $k \geq k_1$.

  We now claim for all $k \geq \max \{ k_0, k_1 \}$ that if we set $M := \kappa \cdot n_k^\beta$
  and $M' := n_k^\theta$, then
  $\big\| \frac{M}{M'} \, \zeta_{M'} \big\|_{\ApproxSpace([0,1]^d)} \leq 1$.
  This will prove the claim, since
  $\big\| \frac{M}{M'} \, \zeta_{M'} \big\|_{C([0,1]^d)} \geq \frac{M}{M'} \to \infty$
  as $k \to \infty$; here, we used that $\zeta_{M'}(0) = 1$ and that $\beta > \theta$.

  To see that $\big\| \frac{M}{M'} \, \zeta_{M'} \big\|_{\ApproxSpace([0,1]^d)} \leq 1$,
  note that $\beta \leq \theta + \delta_2 = \gamma$ and hence
  \[
    M
    = \kappa \cdot n_k^\beta
    \leq c \cdot n_k^\gamma
    \leq (\CoeffFunc(n_k))^L \, n_k^{\lfloor L/2 \rfloor}
    ,
  \]
  so that $M \leq C^L \, n_k^{\lfloor \frac{L}{2} \rfloor}$
  for a suitable $0 \leq C \leq \CoeffFunc(n_k)$.
  Thus, \Cref{eq:SobolevTypeEmbeddingNecessaryMainStep}
  shows $\frac{M}{M'} \zeta_{M'} \in \NNSigma[(d+L) n_k]$.
  Now, for $t \in \N$ there are two cases:
  If $t \geq (d+L) \cdot n_k$,
  then $t^\alpha \, d_p (\frac{M}{M'} \, \zeta_{M'}, \NNSigma[t]) = 0 \leq 1$.
  If otherwise $t < (d+L) \cdot n_k$, then the estimate from Step~2 combined with our
  choice of $M$, $M'$ and $\kappa$ shows that
  \[
    t^\alpha \, d_p\bigl(\tfrac{M}{M'} \, \zeta_{M'}, \NNSigma[t]\bigr)
    \leq t^\alpha \, \tfrac{M}{M'} \, \| \zeta_{M'} \|_{L^p([0,1]^d)}
    \leq \kappa \cdot (d+L)^\alpha \cdot n_k^{\alpha - \theta - \theta \frac{d}{p} + \beta}
    \leq n_k^{\alpha - \theta - \theta \frac{d}{p} + \theta + \delta_1}
    =    1 .
  \]
  Finally, we also have
  \(
    \| \frac{M}{M'} \, \zeta_{M'} \|_{L^p([0,1]^d)}
    \leq \frac{M}{M'} \cdot (M')^{-d/p}
    \leq \kappa \cdot n_k^{\beta - \theta - \theta \frac{d}{p}}
    \leq n_k^{\theta + \delta_1 - \theta - \theta \frac{d}{p}}
    =    n_k^{-\alpha}
    \leq 1 .
  \)
  Overall, this shows $\PreNorm(\frac{M}{M'} \, \zeta_{M'}) \leq 1$,
  so that \Cref{lem:ApproximationSpaceProperties} shows
  $\| \frac{M}{M'} \, \zeta_{M'} \|_{\ApproxSpace} \leq 1$ as well.
  As seen above, this completes the proof of the first part of the theorem.

  \medskip{}

  \textbf{Step~4:} In this step, we prove the second part of the theorem.
  By the assumptions of that part, we have
  \(
    \gamma^\ast (\LayerFunc,\CoeffFunc)
    > \alpha \frac{1 - \beta}{\beta + \frac{d}{p}}
    = \frac{\alpha}{\nu}
  \)
  for $\nu := \frac{\beta + \frac{d}{p}}{1 - \beta}$.
  By definition of $\gamma^\ast$, we can thus choose $\gamma > \frac{\alpha}{\nu}$
  and $L \in \N_{\leq \LayerFunc^\ast}$ satisfying
  \[
    c_0
    := \limsup_{n \to \infty}
       \Big[
         (\CoeffFunc(n))^L \cdot n^{\lfloor L/2 \rfloor - \gamma}
       \Big]
    \in (0,\infty].
  \]
  Since $\LayerFunc^\ast \geq 2$, we can clearly assume that $L \geq 2$.
  Choose a subsequence $(n_k)_{k \in \N}$ such that
  $c_0 = \lim_{k \to \infty} (\CoeffFunc(n_k))^L \cdot n_k^{\lfloor L/2 \rfloor - \gamma}$.
  Then, setting $c := \min \{ 1, c_0 / 2 \}$, there exists $k_0 \in \N$ satisfying
  $(\CoeffFunc(n_k))^L \cdot n_k^{\lfloor L/2 \rfloor} \geq c \cdot n_k^\gamma > 0$
  and $L \leq \LayerFunc(n_k)$ for all $k \in \N_{\geq k_0}$.

  Since $\gamma > \frac{\alpha}{\nu} = \frac{1-\beta}{\beta + \frac{d}{p}} \alpha$,
  we have $\frac{\alpha}{\beta + \frac{d}{p}} < \frac{\gamma}{1-\beta}$.
  Hence, we can choose
  \(
    \tau
    \in \bigl(\frac{\alpha}{\beta + \frac{d}{p}}, \frac{\gamma}{1-\beta}\bigr)
    \subset (0,\infty)
    .
  \)
  Then $\tau \cdot (1 - \beta) < \gamma$ and furthermore
  \[
    \tau \cdot (\beta + \tfrac{d}{p}) > \alpha
    \quad \Longrightarrow \quad
    \tau \cdot (1 - \beta - 1 - \tfrac{d}{p}) < - \alpha
    \quad \Longrightarrow \quad
    \tau \cdot (1 - \beta) < \tau \cdot (1 + \tfrac{d}{p}) - \alpha .
  \]
  Hence, we can choose
  \(
    \theta \in
    \big(
      \tau \cdot (1 - \beta), \,\,
      \min
      \{
        \gamma,
        \tau \cdot (1 + \frac{d}{p}) - \alpha
      \}
    \big)
    \subset (0,\infty)
    .
  \)

  Now, choose $\kappa \in (0,1]$ such that $\kappa \leq c$ and $\kappa \cdot (L + d)^\alpha \leq 1$.
  Fix $k \in \N_{\geq k_0}$ for the moment, and set $M := \kappa \cdot n_k^\theta$
  and $M' := n_k^\tau \geq 1$.
  Since
  \(
    (\CoeffFunc(n_k))^L \cdot n_k^{\lfloor L/2 \rfloor}
    \geq c \cdot n_k^\gamma
    \geq \kappa \cdot n_k^\gamma
    ,
  \)
  we can choose $0 \leq C \leq \CoeffFunc(n_k)$ satisfying
  \(
    C^L \cdot n_k^{\lfloor L/2 \rfloor}
    \geq \kappa \cdot n_k^\gamma
    \geq \kappa \cdot n_k^\theta
    = M.
  \)
  Then, Step~1 shows that $f := \frac{M}{M'} \cdot \zeta_{M'} \in \NNSigma[(L+d) n_k]$.
  Now, for $t \in \N$ there are two cases:
  If $t \geq (L + d) \cdot n_k$, then $t^\alpha \cdot d_p (f, \NNSigma[t]) = 0 \leq 1$.
  If otherwise $t < (L+d) \cdot n_k$, then the estimate from Step~2
  combined with the bound $\theta \leq \tau \cdot (1 + \frac{d}{p}) - \alpha$ shows
  \begin{align*}
    t^\alpha \cdot d_p(f, \NNSigma[t])
    & \leq (L+d)^\alpha \cdot n_k^\alpha \cdot \| f \|_{L^p} \\
    & \leq (L+d)^\alpha \cdot n_k^\alpha \cdot \frac{M}{M'} \cdot (M')^{-d/p} \\
    & \leq (L+d)^\alpha \cdot \kappa \cdot n_k^{\alpha + \theta - \tau - \tau \frac{d}{p}} \\
    & \leq n_k^{\alpha + \tau \cdot (1 + \frac{d}{p}) - \alpha - \tau - \tau \frac{d}{p}}
      =    1 .
  \end{align*}
  Furthermore, Step~2 also shows that
  \[
    \| f \|_{L^p}
    = \frac{M}{M'}
      \| \zeta_{M'} \|_{L^p}
    \leq \kappa \cdot n_k^{\theta - \tau} \cdot (M')^{-d/p}
    \leq n_k^{\theta - \tau - \tau \frac{d}{p}}
    \leq n_k^{\tau \cdot (1 + \frac{d}{p}) - \alpha - \tau - \tau \frac{d}{p}}
    =    n_k^{-\alpha}
    \leq 1.
  \]
  Overall, we see that $\PreNorm(f) \leq 1$, so that \Cref{lem:ApproximationSpaceProperties} shows
  $\| f \|_{\ApproxSpace} \leq 1$.

  We now show that $\| f \|_{C^{0,\beta}} \to \infty$ as $k \to \infty$;
  since $\| f \|_{\ApproxSpace} \leq 1$, this will imply that the embedding
  $\ApproxSpace ([0,1]^d) \hookrightarrow C^{0,\beta}([0,1]^d)$ does \emph{not} hold.
  To see that $\| f \|_{C^{0,\beta}} \to \infty$ as $k \to \infty$, note that
  $f(0) = \frac{M}{M'}$ and $f(\frac{1}{M'},0,\dots,0) = 0$.
  Hence,
  \[
    \| f \|_{C^{0,\beta}}
    \geq \Lip_\beta (f)
    \geq \frac{|f(0) - f(\frac{1}{M'},0,\dots,0)|}{\|0 - (\frac{1}{M'},0,\dots,0)\|_{\ell^\infty}^\beta}
    =    (M')^\beta \cdot \frac{M}{M'}
    =    n_k^{\beta \tau - \tau} \cdot \kappa \cdot n_k^\theta
    =    \kappa \cdot n_k^{\theta - \tau \cdot (1 - \beta)}
    \!\xrightarrow[k\to\infty]{}\! \infty ,
  \]
  since $\theta > \tau \cdot (1 - \beta)$.
\end{proof}

\bibliographystyle{myplain}
\footnotesize
\bibliography{references,refs}

\begin{thebibliography}{10}

\bibitem{barron1993universal}
A.~R. Barron.
\newblock Universal approximation bounds for superpositions of a sigmoidal
  function.
\newblock {\em IEEE Transactions on Information theory}, 39(3):930--945, 1993.

\bibitem{beneventano2020high}
P.~Beneventano, P.~Cheridito, A.~Jentzen, and P.~von Wurstemberger.
\newblock High-dimensional approximation spaces of artificial neural networks
  and applications to partial differential equations.
\newblock {\em arXiv preprint arXiv:2012.04326}, 2020.

\bibitem{berner2021modern}
J.~Berner, P.~Grohs, G.~Kutyniok, and P.~Petersen.
\newblock The modern mathematics of deep learning.
\newblock {\em arXiv preprint arXiv:2105.04026}, 2021.

\bibitem{CarageaBarronBoundary}
A.~Caragea, P.~Petersen, and F.~Voigtlaender.
\newblock Neural network approximation and estimation of classifiers with
  classification boundary in a {B}arron class.
\newblock {\em arXiv preprint arXiv:2011.09363}, 2020.

\bibitem{NeuralNetworksApproximatingRefinableFunctions}
I.~Daubechies, R.~DeVore, N.~Dym, S.~Faigenbaum-Golovin, S.~Z. Kovalsky, K.-C.
  Lin, J.~Park, G.~Petrova, and B.~Sober.
\newblock Neural network approximation of refinable functions.
\newblock {\em arXiv preprint arXiv:2107.13191}, 2021.

\bibitem{daubechies2021nonlinear}
I.~Daubechies, R.~DeVore, S.~Foucart, B.~Hanin, and G.~Petrova.
\newblock Nonlinear approximation and (deep) $\mathrm{ReLU}$ networks.
\newblock {\em Constructive Approximation}, pages 1--46, 2021.
\newblock \href{https://doi.org/10.1007/s00365-021-09548-z}{doi:
  10.1007/s00365-021-09548-z}.

\bibitem{devore2020neural}
R.~DeVore, B.~Hanin, and G.~Petrova.
\newblock Neural network approximation.
\newblock {\em Acta Numerica}, 30:327–444, 2021.

\bibitem{DeVoreConstructiveApproximation}
R.~A. DeVore and G.~G. Lorentz.
\newblock {\em Constructive approximation}, volume 303 of {\em Grundlehren der
  Mathematischen Wissenschaften}.
\newblock Springer-Verlag, Berlin, 1993.

\bibitem{weinan2019barron}
W.~E, C.~Ma, and L.~Wu.
\newblock Barron spaces and the compositional function spaces for neural
  network models.
\newblock {\em arXiv preprint arXiv:1906.08039}, 2019.

\bibitem{grohs2019deep}
D.~Elbr{\"a}chter, D.~Perekrestenko, P.~Grohs, and H.~B{\"o}lcskei.
\newblock Deep neural network approximation theory.
\newblock {\em IEEE Transactions on Information Theory}, 67(5):2581--2623,
  2021.

\bibitem{NNApproximationSpaces}
R.~Gribonval, G.~Kutyniok, M.~Nielsen, and F.~Voigtlaender.
\newblock Approximation spaces of deep neural networks.
\newblock {\em Constr. Approx.}, pages 1--109, 2021.
\newblock
  doi:\href{https://doi.org/10.1007/s00365-021-09543-4}{10.1007/s00365-021-09543-4}.

\bibitem{GrohsTheoryToPracticeGap}
P.~Grohs and F.~Voigtlaender.
\newblock Proof of the theory-to-practice gap in deep learning via sampling
  complexity bounds for neural network approximation spaces.
\newblock {\em arXiv preprint arXiv:2104.02746}, 2021.

\bibitem{pinkus1999approximation}
A.~Pinkus.
\newblock Approximation theory of the {MLP} model in neural networks.
\newblock {\em Acta numerica}, 8:143--195, 1999.

\bibitem{RaoRenOrliczSpaces}
M.~M. Rao and Z.~D. Ren.
\newblock {\em Theory of {O}rlicz spaces}, volume 146 of {\em Monographs and
  Textbooks in Pure and Applied Mathematics}.
\newblock Marcel Dekker, Inc., New York, 1991.

\end{thebibliography}

\end{document}